\documentclass[11pt]{article}
\usepackage{amsthm, amsmath}
\usepackage{fullpage}
\usepackage{tikz}
\def\floor#1{\left\lfloor#1\right\rfloor}
\def\ceil#1{\left\lceil#1\right\rceil}
\def\mad{{\rm{mad}}}
\def\lc{{\rm{lc}}}
\def\lcl{{\rm{lc_{\ell}}}}

\newtheorem{theorem}{Theorem}
\newtheorem*{theoremA}{Theorem A}
\newtheorem*{theoremB}{Theorem B}
\newtheorem{lemma}{Lemma}
\newtheorem{cor}{Corollary}
\newtheorem{observation}{Observation}

\begin{document}
\author{Daniel W. Cranston\thanks{Department of Mathematics and Applied Mathematics, Virginia Commonwealth University, Richmond, VA, 23284. email: \texttt{dcranston@vcu.edu}} \and 
Gexin Yu\thanks{Department of Mathematics, College of William \& Mary, Williamsburg, VA, 23185. email: \texttt{gyu@wm.edu}.
 Research supported in part by the NSF grant DMS-0852452.}}

\title{Linear Choosability of Sparse Graphs}
\maketitle

\abstract{
We study the linear list chromatic number, denoted $\lcl(G)$, of sparse graphs.
The maximum average degree of a graph $G$, denoted $mad(G)$,  is the maximum of the average degrees of all subgraphs of $G$. 
It is clear that any graph $G$ with maximum degree $\Delta(G)$ satisfies $\lcl(G)\ge \lceil\Delta(G)/2\rceil+1$.  In this paper, we prove the following results:  (1) if $\mad(G)<12/5$ and $\Delta(G)\ge 3$, then $\lcl(G)=\lceil\Delta(G)/2\rceil+1$,  and we give an infinite family of examples to show that this result is best possible;  
(2) if $\mad(G)<3$ and $\Delta(G)\ge 9$, then $\lcl(G)\le\lceil \Delta(G)/2\rceil+2$, and  we give an infinite family of examples to show that the bound on $\mad(G)$ cannot be increased in general;  
(3) if $G$ is planar and has girth at least 5, then $\lcl(G)\le\lceil\Delta(G)/2\rceil+4$.
}

\section{Introduction}

In 1973, Gr\"{u}nbaum introduced {\it acyclic colorings}~\cite{G}, which are proper colorings with the additional property that each pair of color classes induces a forest.  
In 1997, Hind, Molloy, and Reed introduced frugal colorings~\cite{HMR}.  A proper coloring is {\it $k$-frugal} if the subgraph induced by each pair of color classes has maximum degree less than $k$.  Yuster~\cite{Y} combined the ideas of acyclic coloring and $3$-frugal coloring in the notion of a {\it linear coloring}, which is a proper coloring such that each pair of color classes induces a union of disjoint paths---also called a {\it linear forest}.
We write $\lc(G)$ to denote the {\it linear chromatic number} of $G$, which is the smallest integer $k$ such that $G$ has a proper $k$-coloring in which every pair of color classes induces a linear forest. 

We begin by noting easy upper and lower bounds on $\lc(G)$.  If $G$ is a graph with maximum degree $\Delta(G)$, then we have the naive lower bound $\lc(G)\ge\ceil{\Delta(G)/2}+1$, since each color can appear on at most two neighbors of a vertex of maximum degree.  
Observe that $\lc(G)\le \chi(G^2)\le\Delta(G^2)+1\le \Delta(G)^2+1$, where $\chi(G)$ denotes the chromatic number of $G$ and $G^2$ is the square graph of $G$. 
Yuster~\cite{Y} constructed an infinite family of graphs such that $\lc(G)\ge C_1\Delta(G)^{3/2}$, for some constant $C_1$.  He also proved an upper bound of $\lc(G)\le C_2\Delta(G)^{3/2}$, for some constant $C_2$ and for sufficiently large $\Delta(G)$.

Note that trees with maximum degree $\Delta(G)$ have linear chromatic number $\ceil{\Delta(G)/2}+1$, i.e., the naive lower bound holds with equality (for example, we can color greedily in order of a breadth-first search from an arbitrary vertex).  This equality for trees suggests that sparse graphs might have linear chromatic number close to the naive lower bound.    To be more precise: Is it true that sparse graphs have $\lc(G)\le\ceil{\Delta(G)/2}+C$, for some constant $C$?
To state the previous results related to this question, we first introduce some more notation.


We start with linear list colorings, which are linear colorings from assigned lists.
Formally, let $\lcl(G)$ be the {\it linear list chromatic number} of $G$, that is, the smallest integer $k$ such that if each vertex $v\in V(G)$ is given a list $L(v)$ with $|L(v)|\ge k$, then $G$ has a linear coloring such that each vertex $v$ gets a color $c(v)$ from its list $L(v)$.  When all the lists are the same, linear list coloring is the same as linear coloring.   General list coloring was first introduced by Erd\"{o}s, Rubin, and Taylor~\cite{ERT} and independently by Vizing~\cite{V} in the 1970s, and it has been well-explored since then~\cite{JT}.     

Linear list colorings were first studied by Esperet, Montassier, and Raspaud~\cite{EMR}.  The {\it maximum average degree} of a graph $G$, denoted $\mad(G)$, is the maximum of the average degrees of all of its subgraphs, i.e., $\mad(G)=\max_{H\subseteq G}\frac{2|E(H)|}{|V(H)|}$.  
Observe that the family of all trees is precisely the set of connected graphs with $\mad(G)<2$ (so indeed we are generalizing our motivating example, trees). 
The following results were shown in~\cite{EMR}:

\begin{theoremA}[\cite{EMR}] Let $G$ be a graph:\\
(1)~If $\mad(G)<8/3$, then $\lcl(G)\le\ceil{\Delta(G)/2}+3$.\\
(2)~If $\mad(G)<5/2$, then $\lcl(G)\le\ceil{\Delta(G)/2}+2$.\\
(3)~If $\mad(G)<16/7$ and $\Delta(G)\ge 3$, then $\lcl(G)=\ceil{\Delta(G)/2}+1$.
\end{theoremA}



The {\it girth} of a graph $G$, denoted $g(G)$, or simply $g$, is the length of its shortest cycle.  By an easy application of Euler's formula, 
we see that every planar graph $G$ with girth $g$ satisfies $\mad(G)<2g/(g-2)$.    So we can obtain some results on planar graphs from the above results.   
Raspaud and Wang~\cite{RW} proved somewhat stronger results for planar graphs.

\begin{theoremB}[\cite{RW}] Let $G$ be a planar graph:\\
(1)~If $g(G)\ge 5$, then $\lc(G)\le\ceil{\Delta(G)/2}+14$.\\
(2)~If $g(G)\ge 6$, then $\lc(G)\le\ceil{\Delta(G)/2}+4$\\
(3)~If $g(G)\ge 13$ and $\Delta(G)\ge 3$, then $\lc(G)=\ceil{\Delta(G)/2}+1$.
\end{theoremB}

Our goal in the paper is to improve the results in the above two theorems.   We prove the following: 

\begin{theorem}\label{main}
Let $G$ be a graph:\\
(1)~If $G$ is planar and has $g(G)\ge 5$, then $\lcl(G)\le\ceil{\Delta(G)/2}+4$.\\
(2)~If $\mad(G)<3$ and $\Delta(G)\ge 9$, then $\lcl(G)\le\ceil{\Delta(G)/2}+2$.\\
(3)~If $\mad(G)<12/5$ and $\Delta(G)\ge 3$, then $\lcl(G)=\ceil{\Delta(G)/2}+1$.
\end{theorem}

Raspaud and Wang~\cite{RW} conjectured that the bound in Theorem~\ref{main}(2) holds for all planar graphs with girth at least 6.  Since every such graph $G$ has $\mad(G)<3$, our result proves their conjecture for graphs with $\Delta(G)\ge 9$.   Since $\mad(K_{3,3})=3$ and $\lc(K_{3,3})=5$,  we can construct an infinite family of sparse graphs $G$ containing $K_{3,3}$ such that $\mad(G)=3$, $\Delta(G)=4$, and $lc(G)>\ceil{\Delta(G)/2}+2$.   Thus, the maximum degree condition in Theorem~\ref{main}(2)  cannot be lower than $5$. 

We also note that $lc(K_{2,3})=4>\ceil{\Delta(K_{2,3})/2}+1$ and $\mad(K_{2,3})=12/5$.  Thus, we can construct an infinite family of sparse graphs containing $K_{2,3}$ with maximum degree at most $4$.  All such graphs have $lc(G)=\ceil{\Delta(G)/2}+2$ and can be made sparse enough so that $\mad(G)=\mad(K_{2,3})=12/5$.  So the bound on $\mad(G)$ in Theorem~\ref{main}(3) is sharp.


The proofs of our three results all follow the same outline.  First we prove a structural lemma; this says that each graph under consideration must contain at least one from a list of ``configurations''.  Second, we prove that any minimal counterexample to our theorem cannot contain any of these configurations.  In this second step we begin with a linear list coloring of part of the graph, and show how to extend it to the whole graph.
As we extend the coloring, we often say that we ``choose $c(v)\in L(v)$''; by this we mean that we pick some color $c(v)$ from $L(v)$ and use $c(v)$ to color vertex $v$.    In the following three sections,  we will prove our three main results, respectively.  

 For convenience, we introduce the following notation.   A {\it $k$-vertex} is a vertex of degree $k$.  A $k^+$-vertex ($k^-$-vertex) is a vertex of degree at least (at most) $k$.  A {$k$-thread} is a path of $k+2$ vertices, where each of the $k$ internal vertices have degree 2, and each of the end vertices have degree at least 3.

\section{Planar with girth at least 5 implies $\lcl(G)\le\ceil{\frac{\Delta(G)}2}+4$}
\label{sec2}
\begin{lemma}
\label{girth5lemma}
If $G$ is a planar graph with $\delta(G)\ge 2$ and with girth at least $5$,
then $G$ contains one of the following two configurations:
\begin{enumerate}

\item[(RC1)] a 2-vertex adjacent to a $5^-$-vertex,

\item[(RC2)] a 5-face with four incident 3-vertices and the fifth incident vertex 
of degree at most 5.
\end{enumerate}
\end{lemma}
\begin{proof}
We use the discharging method, with initial charge $\mu(f)=d(f)-5$ for each face $f$ and initial charge $\mu(v) = \frac32d(v)-5$ for each vertex $v$.  
By Euler's formula, we have $\sum_{v\in V(G)}\mu(v)+\sum_{f\in F(G)}\mu(f)
= (3|E|-5|V|)+(2|E|-5|F|)
=-5(|F|-|E|+|V|)
=-10$.
We redistribute charge via the following two discharging rules:
\begin{enumerate}
\item[(R1)] Each $4^+$-vertex $v$ sends charge $\frac{\frac32d(v)-5}{d(v)}$ to each incident face.
\item[(R2)] Each face sends charge 1 to each incident 2-vertex and charge $\frac16$ 
to each incident 3-vertex.
\end{enumerate}

Now we will show that if $G$ contains neither configuration (RC1) nor (RC2), then each vertex and each face finishes with nonnegative charge. This is a contradiction, since the discharging rules preserve the sum of the charges (which begins negative).
We write $\mu^*(v)$ and $\mu^*(f)$ to denote the charge at vertex $v$ or face $f$ after we apply all discharging rules.
If $d(v)=2$, then $\mu^*(v) = (\frac32(2)-5)+2(1) = 0$.  If $d(v)=3$, then $\mu^*(v)=(\frac32(3) -5) + 3(\frac16)=0$.  By design, each $4^+$-vertex finishes with charge 0.  So, we now consider the final charge on each face.

Let $f$ be a face of $G$.  For each pair, $u_1$ and $u_2$, of adjacent vertices on $f$, we compute the net charge given from $f$ to $u_1$ and $u_2$.  If neither of $u_1$ and $u_2$ is a 2-vertex, then each vertex receives charge at most $\frac16$ from $f$, so the net charge given from $f$ to $u_1$ and $u_2$ is at most $2(\frac16)=\frac13$.  If one of $u_1$ and $u_2$, say $u_1$, is a 2-vertex, then, since $G$ does not contain (RC1), we have $d(u_2)\ge 6$.  Hence, the net charge given from $f$ to $u_1$ and $u_2$ is at most $1-\frac23=\frac13$.  (This is true because as the degree of a vertex increases beyond 6, the charge it gives to each incident face increases beyond $\frac23$.) By a simple counting argument, we see that the net total charge given from $f$ to all incident vertices is at most $\frac12(\frac13d(f))=\frac16d(f)$.  Since $\mu(f)=d(f)-5$, we see that $\mu^*(f)\ge 0$ when $d(f)\ge 6$.  Now we consider 5-faces.

Suppose $f$ is a 5-face.  Let $n_2$, $n_3$, and $n_{6^+}$ denote the number of 2-vertices, 3-vertices, and $6^+$-vertices incident to $f$.  Note that $\mu^*(f)\ge -n_2-\frac16n_3+\frac23n_{6^+}$.  From (RC1), we have $n_2\le \floor{d(f)/2}=2$.  If $n_2=2$, then $n_3=0$ and $n_{6^+}=3$, so $\mu^*(v)\ge -2 -\frac16(0)+ \frac23(3)=0$.
If $n_2=1$, then $n_{6^+}\ge 2$, so $n_3\le 2$.  Hence, $\mu^*(f)\ge -1-\frac16(2)+\frac23(2)=0$.  

Suppose now that $f$ is a 5-face and $n_2=0$.  Since we have no copy of (RC2), we have either $n_3=4$ and $n_{6^+}=1$, or we have $n_3\le 3$.  In the first case, we get $\mu^*(f)\ge-0-\frac16(4)+\frac23(1)=0$.
In the second case, note that $f$ has at least two $4^+$-vertices, each of which gives $f$ charge at least $\frac14$.  
Thus $\mu^*(f)\ge -0-\frac16(3)+\frac14(2)=0$.  Hence, every face and every vertex has nonnegative charge.  This contradiction completes the proof.
\end{proof}

In Sections~\ref{sec3} and~\ref{sec4}, we will only assume bounded maximum average degree (rather than planarity and a girth bound).  However, in the proof of the preceeding lemma, we needed the stronger hypothesis of planar with girth at least 5.  Specifically, we used this hypothesis when considering 5-faces.  Our proof relied heavily on the fact that for a 5-face $f$ we have $n_2\le\floor{d(f)/2} < d(f)/2$.

Now we use Lemma~\ref{girth5lemma} to prove the following linear list coloring result, which immediately implies Theorem~1(1).
For technical reasons, we phrase all of our theorems in terms of an integer $M$ such that $\Delta(G)\le M$.  (Without this technical strengthening, when we consider a subgraph $H$ such that $\Delta(H)<\Delta(G)$, we get complications.)  Of course, the interesting case is when $M=\Delta(G)$.

\begin{theorem}
Let $M$ be an integer.  If $G$ is a planar graph with $\Delta(G)\le M$ and girth at least 5, then $\lcl(G)\le\ceil{\frac{M}2}+4$.
\end{theorem}
\begin{proof}
Suppose the theorem is false.  
Let $G$ be a minimal counterexample and let the list assignment $L$ of size $\ceil{\frac{M}2}+4$ be such that $G$ has no linear list coloring from $L$.  
Note that $G$ must be connected.  Suppose $G$ has a 1-vertex $u$ with neighbor $v$.  By minimality, $G-u$ has a linear list coloring from $L$.  Let $L'(u)$ denote the list of colors in $L(u)$ that neither appear on $v$, nor appear twice in  $N(v)$.  Note that $|L'(u)|\ge (\ceil{\frac{M}2}+4) - (\floor{\frac{M-1}2}+1) = 4$.  Thus, if $G$ has a 1-vertex $u$, we can extend a linear list coloring of $G-u$ to $G$.  So we may assume that $\delta(G)\ge 2$.
Since $G$ is a planar graph with $\delta(G)\ge 2$ and girth at least 5, $G$ contains one of the two configurations specified in Lemma~\ref{girth5lemma}.  
\bigskip

{\bf Case (RC1):}
First, suppose that $G$ contains a 2-vertex $u$ adjacent to a $5^-$-vertex $v$.
Let $w$ be the other neighbor of $u$.
By minimality, $G-u$ has a linear list coloring from $L$.
In order to avoid creating any 2-colored cycles and to also avoid creating any vertices that have three neighbors with the same color, it is sufficient to avoid coloring $u$ with any color that appears two or more times in $N(v)\cup N(w)$.  Furthermore, $u$ must not receive a color used on $v$ or on $w$.
 Let $L'(u)$ denote the list of colors in $L(u)$ that may still be used on $u$.
We have $|L'(u)|\ge  (\ceil{\frac{M}2}+4) - (\floor{\frac{(M-1)+(5-1)}2}+2) =(\ceil{\frac{M}2}+4)-(\ceil{\frac{M}2}+3)=1$.  Thus, we can extend a linear list coloring of $G-u$ to a linear list coloring of $G$.
\bigskip

{\bf Case (RC2):}
Suppose instead that $G$ contains a 5-face $f$ with four incident 3-vertices and with the fifth incident vertex of degree at most 5.
We label the vertices as follows:
let $u_1$, $u_2$, $u_3$, and $u_4$ denote successive 3-vertices, and let $v_2$ and $v_3$ denote the neighbors of $u_2$ and $u_3$ not on $f$. 

By minimality, $G-\{u_2,u_3\}$ has a linear list coloring from $L$.
Now we will extend the coloring to $u_2$ and $u_3$.  
Let $L'(u_2)$ and $L'(u_3)$ denote the colors in $L(u_2)$ and $L(u_3)$ that are still available for use on $u_2$ and $u_3$.
When we color $u_2$, we clearly must avoid the colors on $u_1$ and $v_2$.  We also want to avoid creating a 2-colored cycle or a vertex that has three neighbors with the same color.  To do this, it suffices to avoid any color that appears on two or more vertices at distance two from $u_2$.  This gives us an upper  bound on the number of forbidden colors: $2+\floor{\frac{(M-1)+2+2}2}=\ceil{\frac{M}2}+3$.  So $|L'(u_2)|=\ceil{\frac{M}2}+4-(\ceil{\frac{M}2}+3)\ge 1$.  An analagous count shows that $|L'(u_3)|\ge 1$.  However, we might have $L'(u_2)=L'(u_3)$.  Thus, we now refine this argument to show that $|L'(u_2)|\ge 2$ or $|L'(u_3)|\ge 2$.

First suppose that $c(u_1)=c(v_2)$. 
Since the colors on $u_1$ and $v_2$ are the same, these two vertices only forbid a single color from use on $u_2$, rather than the two colors we accounted for above.  Thus we get $|L'(u_2)|\ge 2$.  As above, $|L'(u_3)|\ge 1$, so we first color $u_3$, then color $u_2$ with a color not on $u_3$.  This gives the desired linear coloring of $G$.  Hence, we conclude that $c(u_1)\ne c(v_2)$.  

Since $c(u_1)\ne c(v_2)$, when we color $u_3$, we need not fear creating three neighbors of $u_2$ with the same color.  Further, we need not worry about giving $u_3$ the same color as either $u_1$ or $v_2$, for the following reason.  Any 2-colored cycle that contains $u_3$ and either $u_1$ or $v_2$ must also contain $u_2$ and either $u_4$ or $v_3$.  Thus, by requiring that $u_2$ not get a color that appears on two or more vertices at distance two, we avoid such a 2-colored cycle.  
So in fact, $u_3$ only needs to avoid colors that appear on $v_3$, on $u_4$, or on at least two vertices of $N(u_4)\cup N(v_3)$.  This observation gives us $|L'(u_3)|\ge (\ceil{\frac{M}2}+4) - (\floor{\frac{(M-1)+2}2}+2) =(\ceil{\frac{M}2}+4)-(\ceil{\frac{M}2}+2)=2$.  So we can color $u_2$, then color $u_3$ with a color not on $u_2$.  This gives the desired linear list coloring, and completes the proof.
\end{proof}

A similar, but more detailed, argument proves that if $G$ is a planar graph with girth at least 5 and $\Delta(G)\ge 15$, then $lcl(G)\le\ceil{\frac{\Delta(G)}2}+3$.  A brief sketch of this proof is as follows.  First, we can refine Lemma~\ref{girth5lemma} to show that if $\Delta(G)\ge 15$, then in (RC2) at most two neighbors of $u_1$, $u_2$, $u_3$, and $u_4$ can have high degree.  (The key insight is that our present argument only requires that each $6^+$-vertex give charge $\frac23$ to each incident face; not charge $(\frac32d(v)-5)/d(v)$.  Thus, these high degree vertices have lots of extra charge that they can send to adjacent 3-vertices.)  With a more careful analysis, we can show that both the original configuration (RC1) and this strengthened version of (RC2) are reducible even with only $\ceil{\frac{\Delta(G)}2}+3$ colors.

\section{$\mad(G)<3$ and $\Delta(G)\ge 9$ imply $\lcl(G)\le\ceil{\frac{\Delta(G)}2}+2$}
\label{sec3}

\begin{lemma}
\label{lemma3}
If $G$ is a graph with $\mad(G)<3$, $\delta(G)\ge 2$, and $\Delta(G)\ge 9$,
then $G$ contains one of the following five configurations:
\bigskip

(RC1) a 2-vertex $u$ adjacent to vertices $v$ and $w$ such that $\ceil{\frac{d(v)}2}+\ceil{\frac{d(w)}2}<\ceil{\frac{\Delta(G)}2}+2$,
\smallskip

(RC2) a 3-vertex $u$ adjacent to a 2-vertex and to two other vertices $v$ and $w$, such that $d(v)+d(w)\le8$,
\smallskip

(RC3) a 3-vertex adjacent to two 2-vertices,
\smallskip

(RC4) a 4-vertex adjacent to four 2-vertices,
\smallskip

(RC5) 
a 5-vertex $u$ that is adjacent to four 2-vertices, each of which is adjacent to another $8^-$-vertex; and $u$ is also adjacent to a fifth $3^-$-vertex.
%
\bigskip

\end{lemma}
In fact, the hypothesis $\Delta(G)\ge 9$ cannot be omitted (though the lower bound can possibly be reduced), as we show after we prove the lemma.

\begin{proof}
We use discharging, with initial charge $\mu(v) = d(v)-3$ for each vertex $v$. 
Since $\mad(G)<3$, the sum of the initial charges is negative.
Note that only the $2$-vertices have negative charge, so we design our discharging rules to pass charge to the 2-vertices.  
We redistribute the charge via the following three discharging rules:
\smallskip

\begin{enumerate}
\item[(R1)] Every $4$-vertex gives charge $\frac13$ to each adjacent $2$-vertex.
\item[(R2)] Every $5$-vertex 
gives charge $\frac37$ to each adjacent 2-vertex that is also adjacent to another $8^-$-vertex; 
and it gives charge $\frac5{14}$ to every adjacent 3-vertex and every other adjacent 2-vertex.
\item[(R3)] Every $6^+$-vertex $v$ gives charge $\frac{d(v)-3}{d(v)}$ to each adjacent $2$-vertex and $3$-vertex.
\item[(R4)] Every $3$-vertex gives its charge (that it received from rules (R2) and (R3)) to its adjacent $2$-vertex (if it has one). 
\end{enumerate}

We will show that if $G$ contains none of the five configurations (RC1)--(RC5), then each vertex finishes with nonnegative charge, which is a contradiction.
The following observation is an immediate corollary of the fact that $G$ contains no copy of (RC1).  We will use this observation below, to show that every vertex finishes with nonnegative charge.

\begin{observation}\label{obs1}
Suppose that a 2-vertex $u$ has neighbors $v$ and $w$.
\begin{enumerate}
\item[(i)] If $d(v)\in \{3, 4\}$, then $d(w)= \Delta(G)$ if $\Delta(G)$ is odd,  and $d(w)\ge \Delta(G)-1$ if $\Delta(G)$ is even. 
\item[(ii)] If $d(v)\in \{5, 6\}$, then $d(w)\ge \Delta(G)-2$ if $\Delta(G)$ is odd, and $d(w)\ge \Delta(G)-3$ if $\Delta(G)$ is even. 
\end{enumerate}
\end{observation}

We now use Observation~\ref{obs1} to show that every vertex finishes with nonnegative charge.
It is clear from (R3) that every $6^+$-vertex finishes with nonnegative charge.  The same is true of 3-vertices.  So we consider 4-vertices, 5-vertices, and 2-vertices.  

Suppose $d(u)=4$.
Since $G$ contains no copy of (RC4), every 4-vertex $u$ is adjacent to at most three 2-vertices.  Thus, we have $\mu^*(u)\ge \mu(u)-3(\frac13)=1-3(\frac13)=0$.

Suppose $d(u)=5$.
If $u$ has two or more neighbors that each receive charge at most $\frac5{14}$ from $u$, then $\mu^*(u)\ge \mu(u) - 3(\frac37)-2(\frac5{14}) = 2-\frac{14}7=0$.
Similarly, if $u$ has one neighbor that receives no charge from $u$, then $\mu^*(u)\ge \mu(u) - 4(\frac37) > 0$.
Hence, we may assume that $u$ sends charge to each neighbor, and that it sends charge $\frac37$ to at least four of its neighbors.
However, this assumption implies that $G$ contains a copy of configuration (RC5), which is a contradiction.

Finally, suppose $d(u)=2$.
Let the neighbors of $u$ be $v$ and $w$. 
Since $\mu(u)=-1$, it suffices to show that $u$ always receives charge at least $1$.  
If $d(v)\ge 6$ and $d(w)\ge 6$, then $v$ and $w$ each give $u$ charge at least $\frac12$.  So we may assume that $d(v)\le 5$.    
Suppose $d(v)=5$. 
Since $\Delta(G)\ge 9$, Observation~\ref{obs1} implies that $d(w)\ge 7$. 
If $d(w)\in \{7,8\}$, then $u$ receives charge at least $\frac37+\frac47=1$.
If $d(w)\ge 9$, then $u$ receives charge at least $\frac5{14}+\frac69>1$.
%

If $d(v)=4$, then Observation~\ref{obs1} implies that $d(w)\ge 9$, so $u$ receives charge at least $\frac{1}{3}+\frac{6}{9}=1$.  
If $d(v)=3$, then the absence of (RC2) implies that at least one neighbor $x$ of $v$ has degree at least $5$, so $v$ receives charge at least $\frac5{14}$ from $x$. 
Since $v$ can have at most one adjacent $2$-vertex, $u$ gets charge at least $\frac5{14}$ from $v$.  Hence, the total charge that $u$ receives is at least $\frac69+\frac5{14}>1$.  
\end{proof}

Now we give two examples to show that the hypothesis $\Delta(G)\ge 9$, in Lemma~\ref{lemma3} above, can not be omitted.  (We do suspect, however, that this hypothesis can be replaced by $\Delta(G)\ge 7$, or perhaps even by $\Delta(G)\ge 5$.)  We first give an example with maximum degree 3.  Let $G$ be the dodecahedron, and let $E$ be a matching in $G$ of size 6, such that every face of $G$ contains one edge of $E$.  Form $\widehat{G}$ from $G$ by subdividing each edge of the matching.  The girth of $\widehat{G}$ is 6, so (by an easy application of Euler's formula), $\mad(\widehat{G})<3$.  Despite having $\mad(\widehat{G})<3$, $\widehat{G}$ does not contain any of the five configurations (RC1)--(RC5) in Lemma~\ref{lemma3}.  Now we give an example with maximum degree 4.  Let $G$ be the octahedron, and let $E$ be a perfect matching in $G$.  Form $\widehat{G}$ from $G$ by subdividing every edge of $G$ except the three edges of $E$.  The average degree of $\widehat{G}$ is $(4\times 6+2\times 9)/(6+9)=\frac{14}5$; it is an easy exercise to verify that $\mad(\widehat{G})=\frac{14}5$.  Again $\widehat{G}$ contains none of configurations (RC1)--(RC5).

Now we use Lemma~\ref{lemma3} to prove the following linear list coloring result, which immediately implies Theorem~1(2).


\begin{theorem}
Let $M\ge 9$ be an integer.  If $G$ is a graph with $\mad(G)< 3$ and $\Delta(G)\le M$, then $\lcl(G)\le \ceil{\frac{M}{2}}+2$.   
\label{thm2}
\end{theorem}
\begin{proof}
Suppose the theorem is false.  
Let $G$ be a minimal counterexample and let the list assignment $L$ of size $\ceil{\frac{M}2}+2$ be such that $G$ has no linear list coloring from $L$.  
Since $M\ge 9$, we have $|L(v)|=\ceil{\frac{M}2}+2\ge 7$ for every $v\in V$.
Note that $G$ must be connected.  Suppose $G$ has a 1-vertex $u$ with neighbor $v$.  By minimality, $G-u$ has a linear list coloring from $L$.
Let $L'(u)$ denote the list of colors in $L(u)$ that neither appear on $v$, nor appear twice in  $N(v)$.  Note that $|L'(u)|\ge (\ceil{\frac{M}2}+2) - (\floor{\frac{M-1}2}+1) = 2$.  Thus, if $G$ has a 1-vertex $u$, we can extend a linear list coloring of $G-u$ to $G$.  So we may assume that $\delta(G)\ge 2$.

Since $G$ is a graph with $\delta(G)\ge 2$ and $\mad(G)<3$, $G$ contains one of the five configurations (RC1)--(RC5) specified in Lemma~\ref{lemma3}.  We consider each of these five configurations in turn, and in each case we construct a linear coloring of $G$ from $L$.
\bigskip

{\bf Case (RC1):} Suppose that $G$ contains configuration (RC1).
Let $u$ be a 2-vertex adjacent to vertices $v$ and $w$ such that $\ceil{\frac{d(v)}2}+\ceil{\frac{d(w)}2}<\ceil{\frac{M}2}+2$.  By the minimality of $G$, subgraph $G-u$ has a linear list coloring $c$.   

If $c(v)\ne c(w)$, then $u$ can receive any color except for $c(v)$, $c(w)$, and those colors that appear twice on $N(v)$ or twice on $N(w)$.    So the number of colors forbidden is at most $2+ \floor{\frac{d(v)-1}{2}}+\floor{\frac{d(w)-1}{2}}= \ceil{\frac{d(v)}{2}}+\ceil{\frac{d(w)}{2}}$.  
Since $|L(u)| = \ceil{\frac{M}2}+2$, and since $\ceil{\frac{d(v)}2}+\ceil{\frac{d(w)}2}<\ceil{\frac{M}2}+2$, we can extend the coloring to $u$.
So we assume instead that $c(v)=c(w)=1$.  

If $c(v)=c(w)$, then (similar to that above), $u$ can receive any color except for $c(v)$ and those colors that appear twice on $N(v)\cup N(w)$.  The number of forbidden colors is at most $1+ \floor{\frac{(d(v)-1)+(d(w)-1)}2}\le \ceil{\frac{d(v)}{2}}+\ceil{\frac{d(w)}{2}}$.  So, once again, we can extend the coloring to $u$.
\bigskip

{\bf Case (RC2):} Suppose that $G$ contains configuration (RC2). 
Let $u$ be a 3-vertex adjacent to a 2-vertex and to two other neighbors $v$ and $w$ with $d(v)+d(w)\le8$.
By the minimality of $G$, subgraph $G-u$ has a linear list coloring from $L$.  If all three neighbors of $u$ have the same color, then we won't get a linear coloring of $G$ no matter how we color $u$.  In this case, we can recolor the 2-vertex and still have a linear coloring of $G-u$.
Now we will extend the coloring to $u$.  

Let $L'(u)$ denote the colors in $L(u)$ that are still available for use on $u$.  When we color $u$, we clearly must avoid the colors on its three neighbors.  We also want to avoid creating a $2$-colored cycle or a vertex that has three neighbors with the same color.  To do this, it suffices to avoid any color that appears on two or more vertices at distance two from $u$.  This gives us an upper bound on the number of forbidden colors: $3+\floor{\frac{(d(v)-1)+(d(w)-1)+1}2}=3+ \floor{\frac{d(v)+d(w)-1}{2}}\le3+\floor{\frac72}=6$.  Since $|L(u)|\ge 7$, we have $|L'(u)|\ge 1$.  Thus, we can extend the coloring to $u$.
\bigskip

\begin{figure}
\begin{center}
\begin{tikzpicture}
[dot/.style={circle,fill=black,minimum size=10pt,inner sep=0pt, outer sep=-1pt},
scale=.90]
\draw (-.25,-2.0) node[label=0:{\hspace{-.8cm}(a): (RC3)}] {};

\filldraw
(0,0) node[dot,label=0:\!$u$]{} -- 
(-1,1) node[dot,label=0:\!$v_1$]{} (-1,1) -- 
(-1,2) node[dot,label=0:\!$w_1$]{} 
(0,0) node[dot,label=0:\!$u$]{} -- 
(0,1) node[dot,label=0:\!$v_2$]{} -- 
(0,2) node[dot,label=0:\!$w_2$]{} 
(0,0) node[dot,label=0:\!$u$]{} -- 
(1,1) node[dot,label=0:\!$v_3$]{};
\draw (1,1) -- (1.25,1.5) (1,1) -- (.75,1.5);
;

\begin{scope}[xshift=5cm, yshift=1cm]
\draw (-.65,-3.0) node[label=0:{\hspace{-.5cm}(b): (RC4)}] {};
\filldraw 
(180:2) node[dot,label=90:$w_1$]{} --
(180:1) node[dot,label=90:$v_1$]{} --
(180:0) node[dot,label=45:$u$]{} 
(90:2) node[dot,label=0:$w_2$]{} --
(90:1) node[dot,label=0:$v_2$]{} --
(90:0)
(0:2) node[dot,label=90:$w_3$]{} --
(0:1) node[dot,label=90:$v_3$]{} --
(0:0)
(270:2) node[dot,label=0:$w_4$]{} --
(270:1) node[dot,label=0:$v_4$]{} --
(0,0);
\begin{scope}[xshift=5.5cm]
\draw (-.35,-3.0) node[label=0:{\hspace{-.8cm}(c): (RC5)}] {};
\filldraw 
(72:2) node[dot,label=right:$w_1$]{} --
(72:1) node[dot,label=right:$v_1$]{} --
(72:0) node[dot,label=45:$u$]{} 
(144:2) node[dot,label=right:$w_2$]{} --
(144:1) node[dot,label=right:$v_2$]{} --
(144:0)
(216:2) node[dot,label=right:$w_3$]{} --
(216:1) node[dot,label=right:$v_3$]{} --
(216:0)
(288:2) node[dot,label=right:$w_4$]{} --
(288:1) node[dot,label=right:$v_4$]{} --
(0,0) (0,0) --
(0:1) node[dot,label=right:$v_5$]{};
\draw (0:1) -- (20:1.5) (0:1) -- (-20:1.5);
\end{scope}
\end{scope}
\end{tikzpicture}
\end{center}
\caption{
Configurations (RC3), (RC4), and (RC5) 
from Lemma~\ref{lemma3} and Theorem~\ref{thm2}.
}
\end{figure}
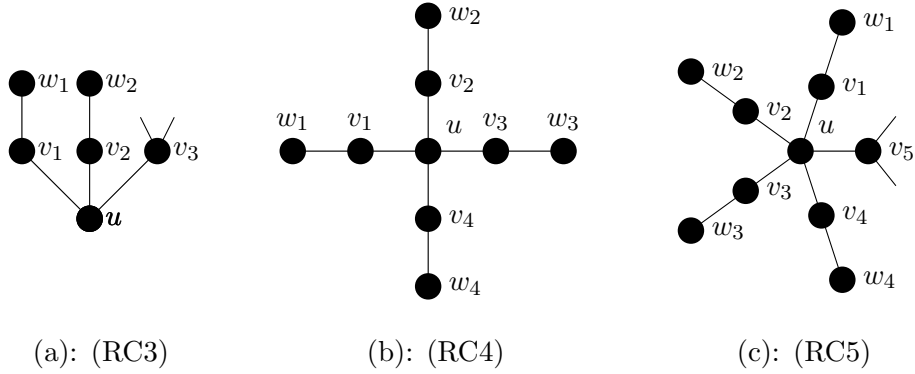

{\bf Case (RC3):}
Suppose that $G$ contains configuration (RC3), shown in Figure~1.
Let $u$ be a $3$-vertex that has neighbors $v_1$, $v_2$, and $v_3$ with $d(v_1)=d(v_2)=2$ and $d(v)=3$.   Let $N(v_i)=\{w_i, u\}$ for $i\in \{1, 2\}$.    By the minimality of $G$, subgraph $G-\{u, v_1, v_2\}$ has a linear list coloring $c$ from $L$.  For each uncolored vertex $z\in\{u,v_1,v_2\}$, let $L'(z)$ denote the colors in $L(z)$ that are still available for use on $z$.  Note that $|L'(z)|\ge 2$ for each uncolored vertex $z$.

Suppose that $L'(u)=\{c(w_1), c(w_2)\}$; this means that $c(v_3)\not\in\{c(w_1), c(w_2)\}$.  Color $u$ with $c(w_1)$.  Now choose $c(v_1)\in L'(v_1)-c(v_3)$ and $c(v_2)\in L'(v_2)-c(w_1)$.  This is a valid linear coloring of $G$. 

Suppose instead that $L'(u)\setminus\{c(w_1), c(w_2)\}\neq\emptyset$.  Choose $c(u)\in L'(u)-\{c(w_1), c(w_2)\}$, choose $c(v_1)\in L'(v_1)-c(u)$, and choose $c(v_2)\in L'(v_2)-c(u)$.   This coloring is proper and contains no $2$-alternating path through $u$.  Hence, it is a linear coloring unless $c(v_1)=c(v_2)=c(v_3)$.   If no other choice of $c(v_1)$ and $c(v_2)$ can avoid this problem, then we can conclude that $L'(v_1)=L'(v_2)=\{c(v_3), c_1\}$ (for some color $c_1$); further $L'(u)-\{c(w_1), c(w_2)\}=\{c_1\}$.  Suppose we are in this case.

If $c(w_1)\ne c(w_2)$, then, without loss of generality, $L'(u)=\{c(w_1),c_1\}$.
Now let $c(u)=c(w_1)$, $c(v_1)=c_1$, and $c(v_2)=c(v_3)$ This is a valid linear coloring.    
So, by relabeling, we may assume that $c(w_1)=c(w_2)=1$, $c(v_3)=2$, and $c_1=3$.  Thus $L'(v_1)=L'(v_2)=\{2, 3\}$ and $L'(u)=\{1, 3\}$. 

Note that $\{2, 3\}\subseteq L'(v_i)$ implies that 2 and 3 each appear at most once in $N(w_i)$ (for $i\in\{1,2\}$).   
If 3 does not appear on both $N(w_1)$ and $N(w_2)$, then let $c(v_1)=c(v_2)=3$ and $c(u)=1$.  
If 2 does not appear on both $N(w_1)$ and $N(w_2)$, then let 
$c(u)=1$, $c(v_1)=2$, $c(v_2)=3$ (or $c(u)=1$, $c(v_1)=3$, $c(v_2)=2$).    
So, we can assume that $2$ and $3$ each appear once on both $N(w_1)$ and $N(w_2)$.   However, now $|L'(v_i)|\ge (\lceil\frac{M}{2}\rceil+2)-(\lfloor\frac{M-3}{2}\rfloor+1)\ge 3$, which is a contradiction. 
\bigskip

{\bf Case (RC4):}
Suppose that $G$ contains configuration (RC4), shown in Figure~1.
Let $u$ be a 4-vertex and let $N(u)=\{v_i: 1\le i\le 4\mbox{ such that } d(v_i)=2\}$.  Also let $N(v_i)=\{u, w_i\}$ for $1\le i\le 4$.   By the minimality of $G$, subgraph $G-\{u,v_1,v_2,v_3,v_4\}$ has a linear list coloring from $L$.  
For each uncolored vertex $z$, let $L'(z)$ denote the list of colors still available for $z$.
Note that $|L'(v_i)|\ge 2$ and $|L'(u)|=|L(u)|=\ceil{\frac{M}{2}}+2\ge 7$, since $M\ge 9$. 

We can color the $v_i$'s from their lists so that every color is used on at most two $v_i$'s, as follows.  If some color $c$ is available for use on two or more $v_i$'s, then use $c$ on exactly two of them, and color each of the remaining $v_i$'s with another color (which could be the same for both of them).  Otherwise, all the $v_i$'s have disjoint lists of available colors, so color them arbitrarily.  

 If the four colors on the $v_i$'s are all distinct, then color $u$ with a fifth color.  
If $c(v_1)=c(v_2)$ but $c(v_1)$, $c(v_3)$, and $c(v_4)$ are all distinct, then choose $c(u)$ so that $c(u)\not\in\{c(v_1), c(v_3), c(v_4), c(w_1)\}$. 
Finally, if $c(v_1)=c(v_2)$ and $c(v_3)=c(v_4)$ (which together imply $c(v_1)\ne c(v_3)$), then choose $c(u)$ so that $c(u)\not\in \{c(v_1), c(v_3), c(w_1), c(w_3)\}$.  
\bigskip

{\bf Case (RC5):}
Suppose that $G$ contains configuration (RC5), shown in Figure~1.
Let $u$ be a 5-vertex and let $N(u)=\{v_i: 1\le i\le 5\}$, such that $d(v_i)=2$ for $1\le i\le 4$ and $d(v_5)\le 3$.  Also let $N(v_i)=\{u, w_i\}$ for $1\le i\le 4$, where $d(w_i)\le 8$.   By the minimality of $G$, subgraph $G- \{u,v_1,v_2,v_3,v_4\}$ has a linear coloring $c$ from $L$.  
For each uncolored vertex $z\in\{u,v_1,v_2,v_3,v_4\}$,
let $L'(z)$ denote the list of colors still available for $z$. 
Since $d(w_i)\le 8$, we have $|L'(v_i)|\ge 3$. 
Conversely, $|L'(u)|\ge \ceil{\frac{M}{2}}+2-(\floor{\frac22}+1)=\ceil{\frac{M}2}\ge 5$, since $M\ge 9$.
Now we let $L''(v_i)=L'(v_i)-c(v_5)$; note that $|L''(v_i)|\ge 2$.  We now extend the coloring by using the lists $L'(u)$ and $L''(v_i)$. 
We can completely ignore $v_5$ (since we deleted $c(v_5)$ from the lists), so the analysis is exactly the same as in Case (RC4).
\end{proof}

As we explained in the introduction, this theorem immediately yields the following corollary.

\begin{cor}
If graph $G$ is planar, has girth at least 6, and $\Delta(G)\ge 9$, then $\lcl(G)\le\ceil{\frac{\Delta(G)}2}+2$.
\end{cor}

Although our proof of Theorem~\ref{thm2} relies heavily on the hypothesis $\Delta(G)\ge 9$, we suspect that the Theorem is true even when this hypothesis is removed.  Namely, we conjecture that every graph $G$ with $\mad(G)<3$ satisfies $\lcl(G)\le\ceil{\frac{\Delta(G)}2}+2$.  If true, this result is best possible, as shown by the graph $K_{3,3}$, since $\lcl(K_{3,3})=5$.  
Furthermore, every graph $G$ with $K_{3,3}\subseteq G$, $\mad(G)=3$, and $\Delta(G)\in\{3,4\}$ shows this result is best possible.

\section{$\mad(G)<\frac{12}5$ implies $\lcl(G)=\ceil{\frac{\Delta(G)}2}+1$}
\label{sec4}

In this section, we prove that if $G$ is a graph with $\Delta(G)\ge 3$ and $\mad(G)<\frac{12}5$, then $\lcl(G)=\ceil{\frac{\Delta(G)}2}+1$.   For such graphs, we prove an upper bound that matches the trivial lower bound $\lcl(G)\ge\ceil{\frac{\Delta(G)}2}+1$.
Recall (from the introduction) that our bound on $\mad(G)$ is best possible, as  demonstrated by $K_{2,3}$, since $\mad(K_{2,3})=\frac{12}5$ and $\lcl(K_{2,3}) > \ceil{\frac{\Delta(K_{2,3})}2}+1$.

\begin{lemma}\label{12-5}
If $G$ is a graph with 
$\mad(G)<\frac{12}5$
and 
$\delta(G)\ge 2$, then $G$ contains one of the following four configurations:
\begin{enumerate}
\item[(RC1)] a $3^+$-thread, 
\item[(RC2)] a 3-vertex $v$ incident to two $1^+$-threads and one 2-thread,
such that the vertex at distance two from $v$ along each $1^+$-thread is a $3^-$-vertex,
\item[(RC3)] adjacent $3$-vertices with at least seven $2$-vertices in their incident threads,
\item[(RC4)] a path of three vertices $uwv$ with $d(u)=d(w)=d(v)=3$ such that $w$ is incident to a $2$-thread and $u$ and $v$ are each incident to two $2$-threads. 
\end{enumerate}
\end{lemma}

\begin{proof}
We use discharging, with initial charge $\mu(v)=d(v)-\frac{12}5$ for each vertex $v$.  Since $\mad(G)<\frac{12}5$, the sum of the initial charges is negative.  We use the following three discharging rules:

\begin{itemize}
\item[(R1)] Every $2$-vertex gets charge $\frac15$ from each of the endpoints of its thread.
\item[(R2)] Every $3$-vertex incident to two $2$-threads gets charge $\frac15$ from its $3^+$-neighbor.
\item[(R3)] Every $3$-vertex incident to a $1$-thread gets charge $\frac15$ from the other endpoint of the $1$-thread if it is a $4^+$-vertex.
\end{itemize}

Now we will show that if $G$ contains none of configurations (RC1)--(RC4), then every vertex finishes with nonnegative charge, which is a contradiction.   
If $d(v)=2$, then $\mu^*(v)=d(v)-\frac{12}5+2(\frac15)=0$.
If $d(v)\ge 4$, then, 
since $G$ contains no $3^+$-threads (by (RC1)),  $v$ gives away charge $\frac15$ to each of at most $2d(v)$ 2-vertices.  
Note further that if $v$ gives away charge $\frac15$ to $t$ 3-vertices via (R2) and/or (R3), for some constant $t$, then $v$ gives away charge $\frac15$ to at most $2d(v)-t$ 2-vertices.
Thus, we have $\mu^*(v)\ge d(v)-\frac{12}5-\frac15(2d(v))=\frac35(d(v)-4)\ge 0$.   So we only need to consider $3$-vertices.  

Let $d(v)=3$.  Suppose $v$ has at most three $2$-vertices in its incident threads.  If $v$ does not give away charge by (R2), then $v$ gives away charge at most $3(\frac15)$, so $\mu^*(v)\ge 3-\frac{12}5 - 3(\frac15)=0$.
If $v$ does give charge by (R2), then, since $G$ contains no copy of (RC3), $v$ has at most two $2$-vertices in its incident threads. 
Thus $v$ gives away charge at most $3(\frac15)$, unless both $v$ is incident to a 2-thread and also $v$ gives away charge by (R2) to two distinct vertices.
However, this situation cannot occur, since it implies that $G$ contains a copy of (RC4), which is a contradiction.

Suppose instead that $v$ has at least four $2$-vertices in its incident threads. Since $G$ contains no copy of (RC2), either $v$ is incident to two $2$-threads and also adjacent to a $3^+$-vertex, or $v$ is incident to two $1$-threads and one $2$-thread and the other end of at least one $1$-thread is a $4^+$-vertex.  In each case, $v$ gives away charge $4(\frac15)$ and receives charge at least $\frac15$, so $\mu^*(v)\ge 3-\frac{12}5-4(\frac15)+\frac15=0$.  
\end{proof} 


Now we use Lemma~\ref{12-5} to prove the following linear list coloring result.

\begin{theorem}
Let $M\ge 3$ be an integer.  If $G$ is a graph with $\mad(G)<\frac{12}5$ and $\Delta(G)\le M$, then $\lcl(G)=\ceil{\frac{M}2}+1$.
\label{thm3}
\end{theorem}

\begin{proof}
Suppose the theorem is false.  Let $G$ be a minimal counterexample and let list assignment $L$, of size $\ceil{\frac{M}2}+1$, be such that $G$ has no linear list coloring from $L$.
Since $M\ge 3$, we have $|L(v)|=\ceil{\frac{M}2}+1\ge3$ for all $v\in V$.
Note that $G$ must be connected.   
Suppose that $G$ contains a $1$-vertex $u$ with neighbor $v$.  By the minimality of $G$, subgraph $G-\{u\}$ has a linear list coloring from $L$.  Let $L'(u)$ denote the list of colors in $L(u)$ that neither appear on $v$ nor appear twice in $N(v)$.
Note that $|L'(u)|\ge (\ceil{\frac{M}2}+1)-\floor{\frac{M-1}2}-1\ge 1$.   Thus, if $G$ has a 1-vertex $u$, we can extend a linear list coloring of $G-u$ to $G$.  So we may assume that $\delta(G)\ge 2$.

Since $G$ has $\delta(G)\ge 2$ and $\mad(G)<\frac{12}5$, $G$ contains one of the four configurations specified in Lemma~\ref{12-5}.  We consider each of these four configurations in turn, and in each case we construct a linear coloring of $G$ from $L$.
\bigskip
 
{\bf Case (RC1):} Suppose that $G$ contains (RC1): a $3^+$-thread.  Let $u, u_1, u_2, u_3, u_4$ be part of the thread, that is, $d(u)\ge 3$, $d(u_1)=d(u_2)=d(u_3)=2$, and $d(u_4)\ge 2$.  By the minimality of $G$, subgraph $G-\{u_2\}$ has a linear coloring from $L$.   If $c(u_1)=c(u_3)$, then $|L(u_2)|\ge 2$, so we choose $c(u_2)\in L(u_2)-\{c(u)\}$.  If $c(u_1)\not=c(u_3)$, then $|L(u_2)|\ge 1$, so we choose $c(u_2)\in L(u_2)$.  Note that either $c(u_2)\not=c(u)$ or $c(u_1)\not=c(u_3)$, so we haven't created a $2$-colored cycle. 
\bigskip

\begin{figure}
\begin{center}
\begin{tikzpicture}
[dot/.style={circle,fill=black,minimum size=10pt,inner sep=0pt, outer sep=-1pt},
scale=.65]
\draw (-.5,-3.7) node[label=0:{\hspace{-.8cm}(a): (RC2)}] {};
\begin{scope}[yshift=-1cm]
\draw
(0,3) 
node[dot,label=left:$u_1''$] {}-- 
(0,2)
node[dot,label=left:$u_1'$] {}-- 
(0,1)
node[dot,label=left:$u_1$] {}-- 
(0,0) 
(0,0)
node[dot,label=left:$u$] {}-- 
(-45:1)
node[dot,label=right:$u_3$] {}-- 
(-45:2)
node[dot,label=right:$u_3'$] {}
(0,0) --
(-135:1)
node[dot,label=left:$u_2$] {}-- 
(-135:2)
node[dot,label=left:$u_2'$] {};
\end{scope}

\begin{scope}[xshift=6cm]
\draw (.25,-3.7) node[label=0:{\hspace{-.8cm}(b): (RC3)}] {};
\filldraw 
(120:3) 
node[dot,label=left:$u_1''$] {}-- 
(120:2) 
node[dot,label=left:$u_1'$] {}-- 
(120:1) 
node[dot,label=left:$u_1$] {}-- 
(0,0) 
node[dot,label=left:$u$] {}-- 
(240:3) 
node[dot,label=left:$u_2''$] {}-- 
(240:2) 
node[dot,label=left:$u_2'$] {}-- 
(240:1) 
node[dot,label=left:$u_2$] {}-- 
(0,0);
\filldraw (0,0) -- (1.25,0) 
node[dot] {};
\begin{scope}[xshift=1.25cm]
\filldraw 
(60:3) node[dot,label=right:$v_1''$] {}-- 
(60:2) node[dot,label=right:$v_1'$] {}-- 
(60:1) node[dot,label=right:$v_1$] {}-- 
(0,0) node[dot,label=right:$v$] {}-- 
(300:2) node[dot,label=right:$v_2'$] {}-- 
(300:1) node[dot,label=right:$v_2$] {}-- 
(0,0) node[dot] {};
\end{scope}

\begin{scope}[xshift=7cm]
\draw (.8,-3.7) node[label=0:{\hspace{-.8cm}(c): (RC4)}] {};
\filldraw 
(120:3) 
node[dot,label=left:$u_1''$] {}-- 
(120:2) 
node[dot,label=left:$u_1'$] {}-- 
(120:1) 
node[dot,label=left:$u_1$] {}-- 
(0,0) 
node[dot,label=left:$u$] {}-- 
(240:3) 
node[dot,label=left:$u_2''$] {}-- 
(240:2) 
node[dot,label=left:$u_2'$] {}-- 
(240:1) 
node[dot,label=left:$u_2$] {}-- 
(0,0) node[dot] {};
\filldraw (0,0) -- (2.5,0) node[dot] {};
\begin{scope}[xshift=1.25cm]
\filldraw 
(90:2.598) node[dot,label=right:$w_1''$] {}-- 
(90:1.732) node[dot,label=right:$w_1'$] {}-- 
(90:.866) node[dot,label=right:$w_1$] {}-- 
(0,0) node[dot,label=30:$w$] {}; 
\end{scope}
\begin{scope}[xshift=2.5cm]
\filldraw 
(60:3) node[dot,label=right:$v_1''$] {}-- 
(60:2) node[dot,label=right:$v_1'$] {}-- 
(60:1) node[dot,label=right:$v_1$] {}-- 
(0,0) node[dot,label=right:$v$] {}-- 
(300:3) node[dot,label=right:$v_2''$] {}-- 
(300:2) node[dot,label=right:$v_2'$] {}-- 
(300:1) node[dot,label=right:$v_2$] {}-- 
(0,0) node[dot] {};
\end{scope}
\end{scope}
\end{scope}
\end{tikzpicture}

\end{center}
\caption{
Configurations (RC2), (RC3),
and (RC4) 
from Lemma~\ref{12-5} and Theorem~\ref{thm3}.
}
\end{figure}
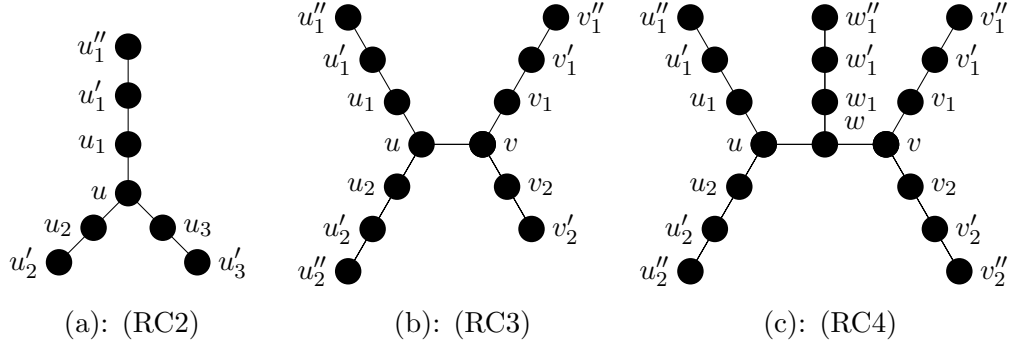

{\bf Case (RC2):}  Suppose instead that $G$ contains (RC2), shown in Figure~2. Let $u$ be a $3$-vertex that is incident to one $2$-thread $u, u_1, u_1', u_1''$ with $d(u_1'')\ge 3$ and incident to two $1^+$-threads $u, u_2, u_2'$ and $u, u_3, u_3'$ 
with $2\le d(u_2')\le 3$ and $2\le d(u_3')\le 3$.
By the minimality of $G$, subgraph $G-\{u, u_1, u_2, u_3\}$ has a linear coloring from $L$. Now we will extend the coloring to $G$.  

For each uncolored vertex $z\in \{u, u_1, u_2, u_3\}$,  let $L'(z)$ denote the colors in $L(z)$ that are still available for use on $z$.   When we extend the coloring, we obviously must get a proper coloring.  In addition, we must avoid creating $2$-colored cycles and avoid creating vertices with $3$ neighbors of the same color.  
Note that $|L'(u_1)|\ge 2$, $|L'(u_2)|\ge 1$, and $|L'(u_3)|\ge 1$.  

Suppose $|L'(u_2)\cup L'(u_3)|\ge 2$.  We choose $c(u_2)\in L'(u_2)$ and $c(u_3)\in L'(u_3)$ such that $c(u_2)\not=c(u_3)$.  Next we choose $c(u)\in L'(u)-\{c(u_2), c(u_3)\}$.  
If $c(u)\not=c(u_1')$, then we choose $c(u_1)\in L'(u_1)-\{c(u)\}$. 
If instead $c(u)=c(u_1')$, then we choose $c(u_1)\in L'(u_1)-\{c(u_1'')\}$.  
This gives a valid linear coloring.

Suppose instead that $|L'(u_2)\cup L'(u_3)|=1$.
Thus $L'(u_2)=L'(u_3)=\{a\}$, for some color $a$.  Clearly, we must choose $c(u_2)=c(u_3)=a$.
Note that this happens only if both $d(u_2')=d(u_3')=3$ and the two other neighbors of $u_2'$ (and $u_3'$) have the same color.  
Now we choose $c(u_1)\in L(u_1)-\{a, c(u_1')\}$ and $c(u)\in L(u)-\{a\}$. 

Since $c(u_1)\not=a$, we haven't created any vertex with 3 neighbors of the same color, and we haven't created any 2-colored cycle passing through $u_1$.  Since $c(u_2)$ does not appear on the other neighbors of $u_2'$,  we haven't created any $2$-colored cycle passing through $u_2$.    
\bigskip 

{\bf Case (RC3):}  Now suppose instead that $G$ contains (RC3): 
two adjacent $3$-vertices with at least seven $2$-vertices in their incident threads (shown in Figure~2). 
We label the vertices as follows:  let $u$ and $v$ be the adjacent $3$-vertices,  $u$ is incident to two $2$-threads $u, u_1, u_1', u_1''$ and $u, u_2, u_2', u_2''$ and $v$ is incident to one $2$-thread $v, v_1, v_1', v_1''$ and one $1^+$-thread $v, v_2, v_2'$.  

By the minimality of $G$,  subgraph $G-\{u, v, u_1, u_2, v_1\}$ has a linear coloring from $L$.  Now we will extend the coloring to $G$.   For each vertex $z\in \{u, v, u_1, u_2, v_1 \}$,  let $L'(z)$ denote the colors in $L(z)$ that are still available for use on $z$.   When we extend the coloring, we obviously must get a proper coloring.  In addition, we must avoid creating $2$-colored cycles and avoid creating vertices with $3$ neighbors of the same color.  Note that $|L'(u_1)|\ge 2$, $|L'(u_2)|\ge 2$, $|L'(v_1)|\ge 2$, $|L'(u)|\ge 3$, and $|L'(v)|\ge 3$; we may assume that equality holds in each case.

Since $|L'(u)|=3>2=|L'(u_1)|$, we can choose $c(u)\in L'(u)-L'(u_1)$.  If $c(u)=c(v_2)$, then choose $c(v_1)\in L'(v_1)-\{c(u)\}$ and $c(v)\in L'(v)-\{c(v_1)\}$. 
If instead $c(u)\ne c(v_2)$, then choose $c(v)\in L'(v)-\{c(u)\}$. 

Now if $c(v)\ne c(v_1')$, then choose $c(v_1')\in L'(v_1)-\{c(v)\}$; if $c(v)=c(v_1')$, then choose $c(v_1')\in L'(v_1)-\{c(v_1'')\}$.  Next, choose $c(u_1)\in L'(u_1)-\{c(v)\}$.  Finally, if $c(u)=c(u_2')$, then choose $c(u_2)\in L'(u_2)-\{c(u_2'')\}$; otherwise, choose $c(u_2)\in L'(u_2)-\{c(u)\}$.

Recall that $c(u_1)\ne c(v)$ and either $c(u)\ne c(v_2)$ or $c(v_1)\ne c(v_2)$; thus, we don't create any vertices with three neighbors of the same color.  By construction, we have no 2-colored cycles through $u_2$ or $v_1$.  Further, $c(u_1)\ne c(v)$, so we don't create any 2-colored cycles.
\bigskip 

{\bf Case (RC4):}  Suppose that $G$ contains (RC4).  We label the vertices as follows: let $u, w, v$ be the path; let $u, u_1, u_1', u_1''$ and $u, u_2, u_2', u_2''$ be the $2$-threads incident to $u$;  let $v, v_1, v_1', v_1''$ and $v, v_2, v_2', v_2''$ be the $2$-threads incident to $v$; and let $w, w_1, w_1', w_1''$ be the $2$-thread incident to $w$.  

By the minimality of $G$,  subgraph $G-\{u, u_1, u_2, v, v_1, v_2, w, w_1\}$ has a linear coloring from $L$.  Now we will extend the coloring to $G$.   For each vertex $z\in \{u, u_1, u_2, v, v_1, v_2, w, w_1\},$  let $L'(z)$ denote the colors in $L(z)$ that are still available for use on $z$.   When we extend the coloring, we obviously must get a proper coloring.  In addition, we must avoid creating $2$-colored cycles and avoid creating vertices with $3$ neighbors of the same color.  We will show explicitly how to color $u$, $u_1$, $u_2$, $w$, and $w_1$ (and we will color $v$, $v_1$, and $v_2$, analogously).
We consider two subcases.  
In fact, we may have one ``side'' ($u$, $u_1$, $u_1'$, $u_2$, and $u_2'$) that is in Subcase (i) and the other side that is in Subcase (ii); this is not a problem, since we color the sides independently.

Subcase (i): 
Suppose that $c(u_1')=c(u_2')$.  If $c(u_1')\not\in L'(u)$, then we can choose $c(u_1)\in L'(u_1)$ and $c(u_2)\in L'(u_2)$ such that $c(u_1)\ne c(u_2)$, and afterward we choose $c(u)\in L'(u)-\{c(u_1'), c(u_1), c(u_2)\}$.  
If $c(u_1')\in L'(u)$, then let $c(u)=c(u_1')$. 
Choose $c(v)$ analogously.
In this instance, we wait to choose $c(u_1)$ and $c(u_2)$ until after we choose $c(w)$.

If $c(u)=c(v)$, then choose $c(w_1)\in L'(w_1)-\{c(u)\}$ and $c(w)\in L'(w)-\{c(w_1), c(u)\}$.  
If $c(u)\ne c(v)$, then choose $c(w)\in L'(w)-\{c(u),c(v)\}$ and $c(w_1)\in L'(w_1)-\{c(w)\}$.
Finally, choose $c(u_1)\in L'(u_1)-\{c(u), c(u_1'')\}$ and $c(u_2)\in L'(u_2)-\{c(u), c(w)\}$ (if we haven't chosen these colors yet; recall that $c(u)=c(u_1')$, so $c(u_1)\ne c(u)$; analogously, $c(u_2)\ne c(w)$).  

Subcase (ii): $c(u_1')\ne c(u_2')$.
Choose $c(u)\in L'(u)-\{c(u_1'), c(u_2')\}$.  Choose $c(v)$ analagously.    
Now color $w$ and $w_1$ as above.
Finally, we will color $u_1$, $u_2$, $v_1$, and $v_2$, as below.   

%
If we can, we choose $c(u_1)\in L'(u_1)-\{c(u)\}$, and $c(u_2)\in L'(u_2)-\{c(u)\}$ such that either $c(u_1)\ne c(w)$ or $c(u_2)\ne c(w)$.  
If this is impossible, then $L'(u_1)=L'(u_2)=\{c(u), c(w)\}$; furthermore, $L(u)=\{c(u), c(u_1'), c(u_2')\}$.  Now let $c(u_1)= c(u_2)=c(u)$ and {\it re-color} $u$ with a new color in $L'(u)-\{c(u_1), c(w_1), c(w)\}$ (note that $c(w)\not\in L'(u)$).    Finally, color $v_1$, $v_2$, and $v$ analogously. 

It is clear that we have created a proper coloring.
It is also straightforward to verify that we didn't create any vertices with $3$ neighbors of the same color, and we didn't create any $2$-colored cycles.  
\end{proof}
This theorem immediately yields the following corollary.

\begin{cor}
If graph $G$ is planar with girth at least $12$ and $\Delta(G)\ge 3$, then $\lcl(G)=\ceil{\frac{\Delta(G)}2}+1$.
\end{cor}


\begin{thebibliography}{99}
%
%
\bibitem{ERT}
P. Erd\"{o}s, A.L. Rubin, H. Taylor, {\it Choosability in graphs}, Congr. Numer., {\bf 26} (1979), pp. 125--157.

\bibitem{EMR}
L. Esperet, M. Montassier, and A. Raspaud, {\it Linear choosability of graphs},  Discrete Math. {\bf 308} (2008),  3938--3950.

\bibitem{G}
B. Gr\"{u}nbaum, {\it Acyclic colorings of planar graphs}, Israel J. Math., {\bf14} (1973), pp. 390--408.

\bibitem{HMR}
H. Hind, M. Molloy, B. Reed, {\it Colouring a graph frugally}, Combinatorica, {\bf 17(4)} (1997), pp. 469--482.

\bibitem{JT}
T.R. Jensen and B. Toft, Graph coloring problems, John Wiley \& Sons, New York, 
1995. 

%
%

\bibitem{RW}
A. Raspaud, W. Wang, {\it Linear coloring of planar graphs with large girth}, Discrete Math. {\bf 309} (2009), pp. 5678--5686.

\bibitem{V}
V.G. Vizing, {\it Coloring the vertices of a graph in prescribed colors} (in Russian), Metody Diskret. Analiz. {\bf 19} (1976), pp. 3--10.


\bibitem{Y}
R. Yuster, {\it Linear coloring of graphs}, Discrete Math. {\bf 185} (1998), pp. 293--297.
\end{thebibliography}
\end{document}